\documentclass[12pt,a4paper]{amsart}

\usepackage{enumerate}
\usepackage{amsmath,amsthm,verbatim,amssymb,amsfonts,amscd,graphicx}
\usepackage{tikz-cd} 
\usepackage{graphics}
\usepackage{hyperref}
\usepackage{mathrsfs}
\usepackage{faktor}
\usepackage[margin=2.9cm]{geometry}
\usepackage[latin2]{inputenc}
\theoremstyle{plain}
\newtheorem{theorem}{Theorem}
\newtheorem*{theorem*}{Theorem}
\newtheorem{corollary}{Corollary}
\newtheorem*{corollary*}{Corollary}
\newtheorem{lemma}{Lemma}
\newtheorem{proposition}{Proposition}

\newtheorem{example}{Example}
\newtheorem*{example*}{Example}
\newtheorem{remark}{Remark}
\theoremstyle{definition}

\begin{document}
\title{Epsilon multiplicity for Noetherian graded algebras}
\author{Suprajo Das}
\address{Suprajo Das, Chennai Mathematical Institute, H1, SIPCOT IT Park, Siruseri, Kelambakkam 603103, India}
\email{dassuprajo@gmail.com}
\maketitle
\begin{abstract}
The notion of epsilon multiplicity was originally defined by Ulrich and Validashti as a limsup and they used it to detect integral dependence of modules. It is important to know if it can be realized as a limit. In this article we show that the relative epsilon multiplicity of reduced Noetherian graded algebras over an excellent local ring exists as a limit.  An important special case of Cutkosky's result concerning epsilon multiplicity, is obtained as a corollary of our main theorem. We also develop the notion of mixed epsilon multiplicity for monomial ideals.
\end{abstract}
\section{Introduction}
The purpose of this paper is twofold: to prove a very general theorem on the relative epsilon multiplicity of graded algebras over a local ring and to develop the notion of mixed epsilon multiplicity for monomial ideals. The idea of $\varepsilon$-multiplicity originates in the works of Kleiman, Ulrich and Validashti.
\subsection{Relative epsilon multiplicity for graded algebras}
Suppose that $(R,m_R)$ is a Noetherian local ring and $$A = \bigoplus\limits_{n\in\mathbb{N}}A_n \subset B = \bigoplus\limits_{n\in\mathbb{N}}B_n$$ is a graded inclusion of standard graded Noetherian $R$-algebras with $A_0 = B_0 = R$. Then the \emph{relative $\varepsilon$-multiplicity of $A$ and $B$} is defined to be 
\begin{align*}
    \varepsilon\left(A\mid B\right) &:= \limsup\limits_{n\to\infty}\dfrac{(\dim B - 1)!}{n^{\dim B - 1}}l_R\left(H^0_{m_R}\left(\dfrac{B_n}{A_n}\right)\right)\\
    &= \limsup\limits_{n\to\infty}\dfrac{(\dim B - 1)!}{n^{\dim B - 1}}l_R\left(\dfrac{A_n\colon_{B_n}m_R^{\infty}}{A_n}\right).
\end{align*}

In \cite{BJ2}, it is proven that this invariant is finite. The question of whether $\varepsilon$-multiplicity actually exists as a limit, has already been considered in several papers such as \cite{DC1}, \cite{DC3}, \cite{DC4}, \cite{DC5}, \cite{DC6}, \cite{BJ2}, \cite{J} and \cite{das}. 

We are ready to state one of the main theorems from our paper.
\begin{theorem*}[Theorem $\ref{maintheorem}$]
Suppose that $(R,m_R)$ is an excellent local ring and $$A = \bigoplus\limits_{n\in\mathbb{N}}A_n \subset B = \bigoplus\limits_{n\in\mathbb{N}}B_n$$ is a graded inclusion of reduced graded Noetherian $R$-algebras with $A_0 = B_0 = R$. Suppose that if $P$ is a minimal prime ideal of $B$ (which is necessarily homogeneous) then $P\cap R \neq m_R$ and if $A_1\subset P$ then $A_n \subset P$ for all $n\geq 1$. Then the limit $$\lim\limits_{n\to\infty}\dfrac{l_R\left(H^0_{m_R}\left(B_n/A_n\right)\right)}{n^{\dim B -1}}$$ exists.
\end{theorem*}
In \cite{das}, the aforementioned result was established by the author, with the additional assumption that $A$ and $B$ are standard graded $R$-algebras. As an application of our main theorem, we show in Corollary \ref{noeth} that it recovers an important special case of Cutkosky's result \cite[Theorem 6.3.]{DC6}.
\subsection{Epsilon multiplicity of an ideal}
Given a Noetherian local ring $(R,m_R)$ of Krull dimension $d$, the \emph{epsilon multiplicity} of an ideal $I\subset R$ is defined to be $$\varepsilon_R(I) := \limsup\limits_{n\to\infty} \dfrac{l_R\left(H^0_{m_R}\left(R/I^n\right)\right)}{n^d}.$$ By considering the graded inclusion of standard graded Noetherian $R$-algebras $$A:= \bigoplus\limits_{n=0}^{\infty}I_nt^n \subset B:= \bigoplus\limits_{n=0}^{\infty} Rt^n,$$ it is not hard to see that $\varepsilon_R(I) = \varepsilon\left(A\mid B\right)$. It can be also verified that if $I$ is an $m_R$-primary ideal then $\varepsilon_R(I)=e_R(I)$, where $e_R(I)$ denotes the Hilbert-Samuel multiplicity of $I$. In \cite{BJ2}, Ulrich and Validashti prove that if the ambient ring $R$ is also equidimensional and universally catenary, then $\varepsilon_R(I)> 0$ if and only if the analytic spread of $I$ is $d$.

A sequence of ideals $\mathcal{I}=\{I_n\}_{n\in\mathbb{N}}$ in a commutative ring $R$ is called a \emph{graded family} if $I_0 = R$ and $\bigoplus_{n\in\mathbb{N}}I_n$ is a graded $R$-algebra. A graded family of ideals $\mathcal{I}=\{I_n\}_{n\in\mathbb{N}}$ is called \emph{Noetherian} if $\bigoplus_{n\in\mathbb{N}}I_n$ is a finitely generated graded $R$-algebra.  A graded family of ideals $\mathcal{I}=\{I_n\}_{n\in\mathbb{N}}$ in $R$ is said to be a \emph{filtration} if $I_0 \supset I_1 \supset I_2 \supset \cdots$.

The result stated below is due to Cutkosky.

\begin{theorem*}\cite[Theorem 6.3.]{DC6}
Suppose that $(R,m_R)$ is an analytically unramified local ring of Krull dimension $d>0$ and $\mathcal{I} = \{I_n\}_{n\in\mathbb{N}}$ is a graded family of ideals in $R$. Further suppose that there exists an integer constant $c>0$ such that 
\begin{equation}\label{lingrow}
 I_n \cap m_R^{cn} = (I_n \colon_R m_R^{\infty})\cap m_R^{cn}
\end{equation}
 for all $n\geq 1$. Assume that if $P$ is a minimal prime ideal of $R$ then $I_1 \subset P$ implies $I_n \subset P$ for all $n\geq 1$. Then the limit $$\lim\limits_{n\to\infty}\dfrac{l_R\left(H^0_{m_R}\left(R/I_n\right)\right)}{n^d}$$ exists.
\end{theorem*}
 A mild modification \cite[Proposition 2.4.]{montano} of Swanson's main theorem in \cite{IS} shows that any Noetherian graded family of ideals satisfy the linear growth condition \eqref{lingrow} of the previous theorem. We provide a simple example showing that the behaviour of the sequence $\left\{l_R\left(H^0_{m_R}\left(R/I_n\right)\right)\right\}_{n\in\mathbb{N}}$ can be erratic if we allow arbitrary graded families $\mathcal{I} = \{I_n\}_{n\in\mathbb{N}}$.
\begin{example*}[Example $\ref{counter}$]
Let $R=K[X,Y]$ be the polynomial ring in two variables over a field $K$ and let $m_R = (X,Y)$ be the graded maximal ideal of $R$. Let $\{a_n\}_{n\in\mathbb{N}}$ be any sequence of natural numbers. Define
\begin{align*}
 I_0 &= R\\
 I_n &= \left(XY^{a_n}, X^2\right) \quad \forall n\geq 1.
\end{align*}
Then $\mathcal{I} = \{I_n\}_{n\in\mathbb{N}}$ is a graded family ideals in $R$ and $$l_R\left(H^0_{m_R}\left(R/I_n\right)\right) = a_n \quad \forall n\geq 1.$$
\end{example*}

Even for $I$-adic filtrations, the sequence $\left\{l_R\left(H^0_{m_R}\left(R/I^n\right)\right)\right\}_{n\in\mathbb{N}}$ can be difficult to predict. A suprising example given in \cite{DC3} shows that there exist an ideal $I$ in a regular local ring $R$ such that $\varepsilon_R(I)$ is an irrational number. This shows that the sequence $\left\{l_R\left(H^0_{m_R}\left(R/I^n\right)\right)\right\}_{n\in\mathbb{N}}$ does not have polynomial growth eventually, unlike the classical Hilbert-Samuel function. However, this sequence is well behaved in the case of monomial ideals. The next theorem is a special case of a result of Herzog, Puthenpurakal and Verma.

A function $\sigma \colon \mathbb{N}^r \to \mathbb{Q}$ is said to be \emph{periodic} if there exists a positive integer $a$ such that $$\sigma(n_1,\ldots,n_{i-1},n_i+a,n_{i+1},\ldots,n_r) = \sigma(n_1,\ldots,n_r)$$ for all $(n_1,\ldots,n_r)\in\mathbb{N}^r$ and $i=1,\ldots,r$.

\begin{theorem*}\cite[Theorem 2.5.]{J}
 Let $R=K[X_1,\ldots,X_d]$ be the polynomial ring in $d$ variables over a field $K$, $m_R$ be the graded maximal ideal of $R$, and $I\subset R$ a monomial ideal. Assume that the analytic spread of $I$ is $d$. Then there exist periodic functions $\sigma_0,\ldots,\sigma_d \colon \mathbb{N} \to \mathbb{Q}$ such that there is an equality $$l_R\left(H^0_{m_R}\left(\dfrac{R}{I^n}\right)\right) = \sum\limits_{i=0}^d \sigma_i(n)n^i$$ for all $n>>0$. Moreover, $\sigma_d$ is a non-zero constant function.
\end{theorem*}
In particular, this shows that the $\varepsilon$-multiplicity of a monomial ideal having maximal analytic spread, is a positive rational number but we cannot always expect it to be an integer (see \cite[Example 6.3.]{JM}). In this article, we produce a multi-graded generalization of the above theorem. 
\subsection{Mixed epsilon multiplicity for monomial ideals}
\begin{theorem*}[Theorem \ref{mainmain}]
 Let $R=K[X_1,\ldots,X_d]$ be the polynomial ring in $d$ variables over a field $K$, $m_R$ be the graded maximal ideal of $R$ and $I_1,\ldots,I_r$ be monomial ideals in $R$. Assume that the analytic spread of $I_1\cdots I_r$ is $d$. Then there exist periodic functions $\sigma_{i_1,\ldots,i_r} \colon \mathbb{N}^r \to \mathbb{Q}$ for all $i_1+\cdots +i_r\leq d$ such that there is an equality $$l_R\left(H^0_{m_R}\left(\dfrac{R}{I_1^{n_1}\cdots I_r^{n_r}}\right)\right) = \sum\limits_{i_1+\cdots+i_r\leq d} \sigma_{i_1,\ldots,i_r}(n_1,\ldots,n_r)n_1^{i_1}\cdots n_r^{i_r}$$ for all $n_1,\ldots,n_r>>0$. Moreover, $\sigma_{i_1,\ldots,i_r}$ is a constant function whenever $i_1+\cdots +i_r = d$ and $\sigma_{i_1,\ldots,i_r}\neq 0$ for some $i_1+\cdots +i_r = d$
\end{theorem*}
Let the assumptions be as in the preceeding theorem. Define $$P(n_1,\ldots,n_r) = \sum\limits_{i_1+\cdots+i_r= d} \sigma_{i_1,\ldots,i_r}(n_1,\ldots,n_r)n_1^{i_1}\cdots n_r^{i_r},$$ which is a homogeneous polynomial in $n_1,\ldots,n_r$ of total degree $d$ with rational coefficients. We define the \emph{mixed epsilon multiplicity} of $R$ of type $(d_1,\ldots,d_r)$ with respect to the monomial ideals $I_1,\ldots,I_r$, denoted by $\varepsilon_R(I_1^{[d_1]},\ldots,I_r^{[d_r]})$, from the coefficients of the homogeneous polynomial $P(n_1,\ldots,n_r)$. Specifically, we write $$P(n_1,\ldots,n_r) = \sum\limits_{d_1+\cdots+d_r=d}\dfrac{1}{d_1!\cdots d_r!} \varepsilon_R(I_1^{[d_1]},\ldots,I_r^{[d_r]}) n_1^{d_1}\cdots n_r^{d_r}.$$ Note that if $r=1$, then $\varepsilon_R(I^{[d]}) = \varepsilon_R(I)$ where $\varepsilon_R(I)$ denotes the usual epsilon multiplicity of $I$. If $I_1,\ldots,I_r$ are also $m_R$-primary, then $$\varepsilon_R(I_1^{[d_1]},\ldots,I_r^{[d_r]}) = e_R(I_1^{[d_1]},\ldots,I_r^{[d_r]})$$ where $e_R(I_1^{[d_1]},\ldots,I_r^{[d_r]})$ denotes the classical mixed multiplicity of $R$ of type $(d_1,\ldots,d_r)$ with respect to the $m_R$-primary monomial ideals $I_1,\ldots,I_r$ (see \cite[Chapter 17.4.]{HS}).  

In fact, far more can be said about the epsilon multiplicity for monomial ideals. The following result is due to Jeffries and Mont\~{a}no.
\begin{theorem*}\cite[Theorem 5.1.]{JM}
 Let $R=K[X_1,\ldots,X_d]$ be the polynomial ring in $d$ variables over a field $K$ and $I\subset R$ a monomial ideal. Then $$\varepsilon_R(I) = d!\mathrm{vol}\left(\mathrm{out}(I)\right).$$
\end{theorem*}
It would be nice to derive analogous formulas in the case of mixed epsilon multiplicity for monomial ideals.

\section{Limits of graded algebras over a local domain}
The following lemma is well-known and is useful for studying Noetherian graded modules over a Noetherian graded ring. We provide a proof because of lack of a proper reference.
\begin{lemma}\label{period}
Suppose that $R$ is a ring, $A = \bigoplus_{n\in\mathbb{N}}A_n$ a finitely generated graded $R$-algebra with $A_0 = R$ and $M = \bigoplus_{n\in\mathbb{N}}M_n$ a finitely generated graded $A$-module. Then there exists an integer $a\geq 1$ such that $$M_{an+r} = A_a^{n-1}M_{a+r}$$ for all $n\geq 1$ and $0\leq r\leq a-1$.
\end{lemma}
\begin{proof}
Say $A$ is generated by $f_1,\ldots,f_t$ as an $R$-algebra and $M$ is generated by $m_1,\ldots, m_s$ as an $A$-module. By replacing $f_i$ and $m_j$ with their homogeneous parts, we may assume $f_i$ is homogeneous of positive degree $d_i$ and $m_j$ is homogeneous of non-negative degree $e_j$. Set 
\begin{align*}
 d &= \mathrm{lcm}\{d_1,\ldots,d_t\},\\
 e &= \max\{e_1,\ldots, e_s\}.
\end{align*}
Let $n$ be a non-negative integer and $r$ an integer satisfying $0\leq r \leq d(t+e)-1$. Then any element of $M_{d(t+e+n)+r}$ is an $R$-linear combination of monomials $f_1^{n_1}\cdots f_t^{n_t}m_j$ such that $$\sum\limits_{i=1}^t n_id_i +e_j = d(t+e+n)+r.$$ Rewriting the expression above, we get $\sum\limits_{i=1}^t n_id_i = d(t+n) + (de-e_j) + r \geq dt$. By an elementary argument we must have that $n_i \geq d/d_i$ for some $i$. Hence every monomial term of $M_{d(t+e+n)+r}$ is a product of a monomial term of $A_{d}$, namely $f_i^{d/d_i}$, and a monomial term of $M_{d(t+e+n-1)+r}$. In other words, we get that $$M_{d(t+e+n)+r} = A_d M_{d(t+e+n-1)+r}$$ for all $n\geq 0$. Now assume that $n\geq 1$ and by repeating the above argument $(t+e)(n-1)$ times, we conclude that $$M_{d(t+e)n+r} = \left(A_{d}\right)^{(t+e)(n-1)}M_{d(t+e)+r} = \left(A_{d(t+e)}\right)^{n-1}M_{d(t+e)+r}.$$ Our claim is now proven if we set $a=d(t+e)$.
\end{proof}
\begin{remark}
The conclusions of Lemma \ref{period}, still hold if the integer $a$ is replaced by any higher multiple of it.
\end{remark}
\begin{lemma}\label{dimless}
Suppose that $(R,m_R)$ is a universally catenary Noetherian local ring and $$A = \bigoplus_{n\in\mathbb{N}}A_n \subset B =\bigoplus_{n\in\mathbb{N}}B_n$$ is a graded inclusion of reduced finitely generated graded $R$-algebras with $A_0 = B_0 = R$. Then $\dim A \leq \dim B$.
\end{lemma}
\begin{proof}
Let $P_1$,\ldots, $P_t$ be the minimal primes (which are necessarily homogeneous) of $B$. For every $1\leq i\leq t$, let $Q_i = P_i \cap A$. Since $B$ is reduced, we know that $\bigcap_{i=1}^t P_i = 0$. So $\left(\bigcap_{i=1}^t P_i\right)\cap A = \bigcap_{i=1}^t Q_i = 0$. This shows that the minimal primes of $A$ appear amongst the primes $Q_1,\ldots,Q_t$. For every $i=1,\ldots,t$, there are graded inclusions $$\dfrac{A}{Q_i} = \bigoplus_{n\in\mathbb{N}}\dfrac{A_n}{Q_i\cap A_n} \subset \dfrac{B}{P_i} = \bigoplus_{n\in\mathbb{N}}\dfrac{B_n}{P_i\cap B_n}$$ of finitely generated graded $R/(P_i\cap R)$-algebras, which are also domains. Let $m_{A/Q_i}$ (respectively $m_{B/P_i}$) denote the homogeneous maximal ideal of $A/Q_i$ (respectively $B/P_i$). As $A/Q_i$ is universally catenary, we obtain from the dimension formula \cite[Theorem $23$, page $84$]{Mat1} that
$$\mathrm{ht}\left(m_{B/P_i}\right) = \mathrm{ht}\left(m_{A/Q_i}\right) + \mathrm{tr.deg.}_{QF\left(A/Q_i\right)} QF\left(B/P_i\right).$$ Thus we get that $\dim B/P_i \geq \dim A/Q_i$ for all $i=1,\ldots,t$. From the definition of Krull dimension, we conclude that $$\dim A = \max_{1\leq i\leq t} \left\{\dim A/Q_i\right\} \leq \max_{1\leq i\leq t} \left\{\dim B/P_i\right\} = \dim B.$$
\end{proof}
\begin{corollary}\label{dimkom}
Suppose that $(R,m_R)$ is a universally catenary Noetherian local ring and $B =\bigoplus_{n\in\mathbb{N}}B_n$ is a reduced finitely generated graded $R$-algebra with $B_0 = R$. Let $a\geq 1$ be an integer such that $B_{an} = B_a^n$ for all $n\geq 0$ (the existence of such an $a$ is guaranteed by Lemma \ref{period}). Define a standard graded Noetherian $R$-algebra $B^{\prime}=\bigoplus_{n\in\mathbb{N}}B^{\prime}_n$ by $B^{\prime}_n = B_a^n = B_{an}$. Then $\dim B^{\prime} \leq \dim B$.
\end{corollary}
\begin{proof}
Define a graded $R$-algebra $B^{\prime\prime}=\bigoplus_{n\geq 0}B^{\prime\prime}_n$ by $B^{\prime\prime}_n = B_n$ if $n$ is divisible by $a$ otherwise $B^{\prime\prime}_n=0$. Note that $B^{\prime\prime}$ is a reduced finitely generated graded $R$-subalgebra of $B$. Also $B^{\prime\prime}$ is isomorphic to $B^{\prime}$ as rings but not as graded rings. We now use Lemma \ref{dimless} to conclude that $\dim B^{\prime} = \dim B^{\prime\prime} \leq \dim B$.
\end{proof}
\begin{lemma}\label{bound}
Suppose that $(R,m_R)$ is a universally catenary Noetherian local ring, $Q\subset R$ an $m_R$-primary ideal, $B=\bigoplus_{n\in\mathbb{N}}B_n$ a reduced finitely generated graded $R$-algebra with $B_0 = R$ and $M=\bigoplus_{n\in\mathbb{N}}M_n$ a finitely generated graded $B$-module. Then $$\limsup\limits_{n\to\infty}\dfrac{l_R\left(M_n/Q^nM_n\right)}{n^{\dim B-1}}<\infty.$$
\end{lemma}
\begin{proof}
We know from Lemma $\ref{period}$ that there exists an integer $a\geq 1$ such that
\begin{align*}
 M_{an+r} &= B_a^{n-1}M_{a+r}\\
 B_{an+r} &= B_a^{n-1}B_{a+r}
\end{align*}
 for all $n\geq 1$ and $0\leq r\leq a-1$. Define a standard graded Noetherian $R$-algebra $C := \bigoplus_{n\geq 0}C_n$ by $C_n = B_{a}^n = B_{an}$. Corollary $\ref{dimkom}$ allows us to conclude that $\dim C \leq \dim B$. Fix an integer $r$ with $0\leq r\leq a-1$. For every integer $n\geq 1$, there is a  short exact sequence $$0 \to \dfrac{C_{n-1}\left(Q^{a+r}M_{a+r}\right)}{Q^{a(n-1)}C_{n-1}\left(Q^{a+r}M_{a+r}\right)} \to \dfrac{C_{n-1}M_{a+r}}{Q^{a(n-1)}C_{n-1}\left(Q^{a+r}M_{a+r}\right)} \to \dfrac{C_{n-1}M_{a+r}}{C_{n-1}\left(Q^{a+r}M_{a+r}\right)} \to 0$$ of finite length $R$-modules. From the additivity of lengths, we obtain
 \begin{equation}\label{test1}
     l_R\left(\dfrac{M_{an+r}}{Q^{an+r}M_{an+r}}\right) = l_R\left(\dfrac{C_{n-1}\left(Q^{a+r}M_{a+r}\right)}{Q^{a(n-1)}C_{n-1}\left(Q^{a+r}M_{a+r}\right)}\right) + l_R\left(\dfrac{C_{n-1}M_{a+r}}{C_{n-1}\left(Q^{a+r}M_{a+r}\right)}\right).
 \end{equation}
Note that $$\dfrac{C}{Q^{a+r}C} = \bigoplus_{n\geq 0}\dfrac{C_n}{Q^{a+r}C_n}$$ is a standard graded Noetherian algebra over the Artinian local ring $R/Q^{a+r}$. Define $$M^{\prime} := \bigoplus_{n\geq 0}\dfrac{C_nM_{a+r}}{Q^{a+r}C_nM_{a+r}},$$ which is a finitely generated graded $C/Q^{a+r}C$-module with generators in degree $0$. From the classical theory of Hilbert series, we conclude that for $n>>0$, the length function $$n \mapsto l_R\left(\dfrac{C_nM_{a+r}}{Q^{a+r}C_nM_{a+r}}\right)$$ is a polynomial in $n$ of degree at most $\dim C/Q^{a+r}C -1\leq \dim C -1\leq \dim B -1$. So the limit
\begin{equation}\label{test2}
    \lim_{n\to\infty}\dfrac{l_R\left(\dfrac{C_{n-1}M_{a+r}}{C_{n-1}\left(Q^{a+r}M_{a+r}\right)}\right)}{(an+r)^{\dim B -1}}
\end{equation}
exists. Define 
\begin{align*}
    H &:= \bigoplus\limits_{(i,j)\in\mathbb{N}^2}\dfrac{Q^{i}C_j}{Q^{i+1}C_j}\\
    M^{\prime\prime} &:= \bigoplus\limits_{(i,j)\in\mathbb{N}^2}\dfrac{Q^iC_j\left(Q^{a+r}M_{a+r}\right)}{Q^{i+1}C_j\left(Q^{a+r}M_{a+r}\right)}.
\end{align*}
 Then $H$ is a standard $\mathbb{N}^2$-graded Noetherian algebra over the Artian local ring $R/Q$ and $M^{\prime\prime}$ is a finitely generated bigraded $H$-module with generators in degree $(0,0)$. Let $a_1,\ldots,a_u$ be the generators of $Q$ as an $R$-module and let $b_1,\ldots,b_v$ be the generators of $C_1$ as an $R$-module. Let $$S = R/Q\left[X_1,\ldots,X_u;Y_1,\ldots,Y_v\right]$$ be a polynomial ring over $R/Q$ and $S$ is bigraded by $\deg X_i = (1,0)$ and $\deg Y_j = (0,1)$. The surjective $R/Q$-algebra homomorphism from $S$ to $H$ which is defined by $$X_i\to [a_i]\in \dfrac{Q}{Q^2},\quad Y_j \to [b_j] \in \dfrac{C_1}{Q C_1}$$ is bigraded. So $H$ can be realized as a finitely generated bigraded $S$-module. Moreover $$H \cong \mathrm{gr}_{QC}C,$$ so that $$\dim_S H = \dim H = \dim \left(\mathrm{gr}_{QC}C\right) = \dim C \leq \dim B.$$ From \cite[Theorem 2.4.]{JO}, it follows that for all $i,j>>0$, the length function $$(i,j) \mapsto l_R\left(\dfrac{Q^iC_j\left(Q^{a+r}M_{a+r}\right)}{Q^{i+1}C_j\left(Q^{a+r}M_{a+r}\right)}\right)$$ is a polynomial in $i$ and $j$ of total degree at most $\dim H -2\leq \dim B -2$. Therefore for $n>>0$, the length function
 $$n \mapsto l_R\left(\dfrac{C_{n}\left(Q^{a+r}M_{a+r}\right)}{Q^{an}C_{n}\left(Q^{a+r}M_{a+r}\right)}\right) = \sum\limits_{i=0}^{an-1}l_R\left(\dfrac{Q^i C_{n}\left(Q^{a+r}M_{a+r}\right)}{Q^{i+1}C_{n}\left(Q^{a+r}M_{a+r}\right)}\right)$$ is a polynomial in $n$ of degree at most $\dim B -1$. It now implies that the limit
 \begin{equation}\label{test3}
     \lim_{n\to\infty}\dfrac{l_R\left(\dfrac{C_{n}\left(Q^{a+r}M_{a+r}\right)}{Q^{an}C_{n}\left(Q^{a+r}M_{a+r}\right)}\right)}{(an+r)^{\dim B-1}}
 \end{equation}
 exists. From the lines $(\ref{test1}), (\ref{test2})$ and $(\ref{test3})$, it follows that the limit $$\lim_{n\to\infty}\dfrac{l_R\left(M_{an+r}/Q^{an+r}M_{an+r}\right)}{(an+r)^{\dim B -1}}$$ exists for all $0\leq r\leq a-1$. Therefore $$\limsup\limits_{n\to\infty}\dfrac{l_R\left(M_n/Q^nM_n\right)}{n^{\dim B -1}} = \max\limits_{0\leq r\leq a-1}\left\{\lim_{n\to\infty}\dfrac{l_R\left(M_{an+r}/Q^{an+r}M_{an+r}\right)}{(an+r)^{\dim B -1}}\right\} <\infty.$$
\end{proof}
Yairon Cid Ruiz and Jonathan Monta\~{n}o have observed in \cite{montano}[Proposition $2.4.$] that Swanson's main result in \cite{IS} can be extended to Noetherian graded family of ideals. We present their observation below.
\begin{theorem}[\cite{montano},\cite{IS}]\label{linear}
Suppose that $(R,m_R)$ is a Noetherian ring and $\mathcal{I} = \{I_n\}_{n\in\mathbb{N}}$ is a Noetherian graded family of ideals in $R$. Then there exists an integer $c>0$ such that for all $n\geq 1$, there exists an irredundant primary decomposition $$I_n = q_1(n)\cap \cdots \cap q_s(n)$$ such that $\left(\sqrt{q_i(n)}\right)^{cn}\subset q_i(n)$ for all $1\leq i\leq s$.
\end{theorem}
\begin{proof}
Let $t$ be an indeterminate and define an $R$-algebra $S$ by
\begin{align*}
    S &= R[t^{-1}, I_1t, I_2t^2, I_3t^3, \ldots]\\
    &=\cdots \oplus Rt^{-3} \oplus Rt^{-2} \oplus Rt^{-1} \oplus R \oplus I_1t \oplus I_2t^2 \oplus I_3t^3 \oplus \cdots
\end{align*}
Note that $S$ is a Noetherian $R$-algebra because $\mathcal{I} = \{I_n\}$ is a Noetherian graded family of ideals in $R$. It can be seen that $(t^{-n})S \cap R = I_n$ and every primary decomposition of $(t^{-n})S$ contracts to a primary decomposition of $I_n$. By replacing $R$ by $S$ and $I_n$ by $(t^{-n})S$, we can assume that $I_n = (x^n)$ where $x$ is a non-zero divisor in a Noetherian ring $R$. Now the same proof of \cite[Theorem 3.4.]{IS} applies.
\end{proof}
The next theorem shows that the relative $\varepsilon$-multiplicity of Noetherian graded algebras is finite under certain conditions.
\begin{theorem}\label{limsup}
Suppose that $(R,m_R)$ is a universally catenary Noetherian local ring and $$A = \bigoplus_{n\in\mathbb{N}} A_n \subset B = \bigoplus_{n\in\mathbb{N}} B_n$$ is a graded inclusion of reduced finitely generated graded $R$-algebras with $A_0 = B_0 = R$. Then $$\limsup\limits_{n\to\infty}\dfrac{l_R\left(H^0_{m_R}\left(A_n/B_n\right)\right)}{n^{\dim B -1}}<\infty.$$
\end{theorem}
\begin{proof}
We first observe that $\{A_nB\}_{n\in\mathbb{N}}$ is a Noetherian graded family of ideals in $R$. By Theorem \ref{linear}, there exists an integer $c_0>0$ such that for all $n\geq 1$, there exists an irredundant primary decomposition $$A_nB = q_1(n)\cap \cdots \cap q_s(n)$$ such that $\left(\sqrt{q_i(n)}\right)^{c_0n}\subset q_i(n)$ for all $1\leq i\leq s$. Suppose that $c\geq c_0$. Since
\begin{align*}
    A_nB\colon_B\, (m_RB)^{\infty} &= \left(\bigcap\limits_{i=1}^t q_i(n)\right)\colon_B\, (m_RB)^{\infty}\\
    &= \bigcap\limits_{i=1}^t\left(q_i(n)\colon_B\,(m_RB)^{\infty}\right)\\
    &= \Big(\bigcap_{\substack{1\leq i\leq t\\ m_RB\not\subset \sqrt{q_i(n)}}}q_i(n)\colon_B\,(m_RB)^{\infty}\Big)\cap \Big(\bigcap_{\substack{1\leq i\leq t\\ m_RB\subset \sqrt{q_i(n)}}}q_i(n)\colon_B\,(m_RB)^{\infty}\Big)\\
    &= \bigcap_{\substack{1\leq i\leq t\\ m_RB\not\subset \sqrt{q_i(n)}}}\left(q_i(n)\colon_B\,(m_RB)^{\infty}\right)\\
    &= \bigcap_{\substack{1\leq i\leq t\\ m_RB\not\subset \sqrt{q_i(n)}}}q_i(n)
\end{align*}
we have that $$(m_RB)^{cn}\cap \left(A_nB\colon_B \,(m_RB)^{\infty}\right) \subset \Big(\bigcap_{\substack{1\leq i\leq t\\ m_RB\not\subset \sqrt{q_i(n)}}}q_i(n)\Big)\cap \Big(\bigcap_{\substack{1\leq i\leq t\\ m_RB\subset \sqrt{q_i(n)}}}q_i(n)\Big) = A_nB.$$
So we get an equality $$(m_RB)^{cn} \cap A_nB = (m_RB)^{cn} \cap \left(A_nB \colon_B (m_RB)^{\infty}\right).$$ In other words $$(m_R^{cn}B_n) \cap A_n = (m_R^{cn}B_n) \cap \left(A_n \colon_{B_n} m_R^{\infty}\right)$$ for all $n\geq 1$. Then $$0 \to \dfrac{A_n}{(m_R^{cn}B_n) \cap A_n} \to \dfrac{\left(A_n \colon_{B_n} m_R^{\infty}\right)}{(m_R^{cn}B_n) \cap \left(A_n \colon_{B_n} m_R^{\infty}\right)} \to \dfrac{A_n \colon_{B_n} m_R^{\infty}}{A_n} \to 0$$ is a short exact sequence of finite length $R$-modules. From length additivity, we get
\begin{align*}
    l_R\left(\dfrac{A_n \colon_{B_n} m_R^{\infty}}{A_n}\right) &= l_R\left(\dfrac{\left(A_n \colon_{B_n} m_R^{\infty}\right)}{(m_R^{cn}B_n) \cap \left(A_n \colon_{B_n} m_R^{\infty}\right)}\right) - l_R\left(\dfrac{A_n}{(m_R^{cn}B_n) \cap A_n}\right)\\
    &\leq l_R\left(\dfrac{\left(A_n \colon_{B_n} m_R^{\infty}\right)}{(m_R^{cn}B_n) \cap \left(A_n \colon_{B_n} m_R^{\infty}\right)}\right)\\
    &\leq l_R\left(\dfrac{B_n}{m_R^{cn}B_n}\right).
\end{align*}
Using Lemma \ref{bound}, we finally deduce that $$\limsup\limits_{n\to\infty}\dfrac{l_R\left(\left(A_n \colon_{B_n} m_R^{\infty}\right)/A_n\right)}{n^{\dim B -1}} \leq \limsup\limits_{n\to\infty} \dfrac{l_R\left(B_n/m_R^{cn}B_n\right)}{n^{\dim B -1}} <\infty.$$
\end{proof}
The next lemma is useful for some reduction arguments needed later.
\begin{lemma}\label{reduction}
 Suppose that $(R,m_R)$ is a universally catenary Noetherian local ring and $$A = \bigoplus_{n\in\mathbb{N}} A_n \subset B = \bigoplus_{n\in\mathbb{N}} B_n$$ is a graded inclusion of reduced finitely generated graded $R$-algebras with $A_0 = B_0 = R$. Also suppose that $P\cap R \neq m_R$ for every minimal prime ideal $P$ of $B$. Then there exists a reduced standard graded Noetherian $R$-algebra $C$ such that $$A = \bigoplus_{n\in\mathbb{N}} A_n \subset B = \bigoplus_{n\in\mathbb{N}} B_n \subset C = \bigoplus_{n\in\mathbb{N}} C_n$$ are graded inclusions of $R$-algebras with $\dim B = \dim C$ and $$\lim\limits_{n\to\infty}\dfrac{l_R\left(H^0_{m_R}\left(C_n/A_n\right)\right)-l_R\left(H^0_{m_R}\left(B_n/A_n\right)\right)}{n^{\dim B -1}}=0.$$
\end{lemma}
\begin{proof}
Since $B$ is a reduced finitely generated graded $R$-algebra, there exists a graded isomorphism $$B \cong \dfrac{R\left[X_1^{d_1},\ldots,X_n^{d_n}\right]}{I}$$ where $X_1,\ldots,X_n$ are variables, $d_1,\ldots,d_n$ are positive integers, $\deg X_i = 1$ for all $1\leq i\leq n$ and $I$ is a homogeneous radical ideal of $R\left[X_1^{d_1},\ldots,X_n^{d_n}\right]$. Observe that $$R\left[X_1^{d_1},\ldots,X_n^{d_n}\right] \subset R\left[X_1,\ldots,X_n\right]$$ is a graded finite extension of Noetherian graded $R$-algebras and using the `going-up theorem' it follows that $$\sqrt{IR\left[X_1,\ldots,X_n\right]} \cap R\left[X_1^{d_1},\ldots,X_n^{d_n}\right] = I.$$ Moreover $\sqrt{IR\left[X_1,\ldots,X_n\right]}$ is a homogeneous radical ideal in $R\left[X_1,\ldots,X_n\right]$. Define $$C := \dfrac{R\left[X_1,\ldots,X_n\right]}{\sqrt{IR\left[X_1,\ldots,X_n\right]}}.$$ Then $C$ is a reduced standard graded Noetherian $R$-algebra and $$B \cong \dfrac{R\left[X_1^{d_1},\ldots,X_n^{d_n}\right]}{I} \subset C = \dfrac{R\left[X_1,\ldots,X_n\right]}{\sqrt{IR\left[X_1,\ldots,X_n\right]}}$$ is a graded integral extension, so that $\dim B = \dim C$. Now apply $H^0_{m_R}$ to the short exact sequence $$0 \to B_n/A_n \to C_n/A_n \to C_n/B_n \to 0$$ and after taking lengths, we obtain
\begin{equation*}
    0 \leq l_R\left(H^0_{m_R}\left(C_n/A_n\right)\right) - l_R\left(H^0_{m_R}\left(B_n/A_n\right)\right) \leq l_R\left(H^0_{m_R}\left(C_n/B_n\right)\right). 
\end{equation*}
Now it is enough to show that $$\lim\limits_{n\to\infty}\dfrac{l_R\left(H^0_{m_R}\left(C_n/B_n\right)\right)}{n^{\dim B-1}}=0.$$ 
From Lemma \ref{period}, it follows that there exists an integer $a\geq 1$ such that
\begin{align*}
    B_{an+r} &= B_a^{n-1}B_{a+r}\\
    C_{an+r} &= B_a^{n-1}C_{a+r}
\end{align*}
for all $n\geq 1$ and $0\leq r\leq a-1$. Define a reduced standard graded Noetherian $R$-algebra $B^{\prime} := \bigoplus_{n\in\mathbb{N}} B^{\prime}_n$ by $B^{\prime}_n = B_a^n = B_{an}$. Corollary $\ref{dimkom}$ allows us to conclude that $\dim B^{\prime} \leq \dim B$. Fix an integer $r$ with $0\leq r\leq a-1$ and define a finitely generated graded $B^{\prime}$-module $$M := \bigoplus_{n\in\mathbb{N}} \dfrac{B_a^n C_{a+r}}{B_a^n B_{a+r}} = \bigoplus_{n\in\mathbb{N}} \dfrac{B^{\prime}_n C_{a+r}}{B^{\prime}_n B_{a+r}},$$ which is generated in degree $0$. Notice that $$H^0_{m_R}(M) = \bigoplus_{n\in\mathbb{N}} H^0_{m_R}\left(\dfrac{B^{\prime}_n C_{a+r}}{B^{\prime}_n B_{a+r}}\right)$$ is a graded $B^{\prime}$-submodule of $M$. In particular, $H^0_{m_R}(M)$ is a finitely generated graded $B^{\prime}$-module. Thus there exists a fixed power $m_R^t$ of $m_R$ that annihilates it and then $H^0_{m_R}(M)$ can be regarded as a finite graded module over $$\dfrac{B^{\prime}}{m_R^tB^{\prime}} = \bigoplus_{n\in\mathbb{N}} \dfrac{B^{\prime}_n}{m_R^t B^{\prime}_n},$$ which is a standard graded algebra over the Artinian local ring $R/m_R^t$. Note that $m_RB^{\prime}$ is not contained in any minimal prime ideal of $B^{\prime}$ becuase $m_RB$ is not contained in any minimal prime ideal of $B$ and minimal prime ideals of $B^{\prime}$ are contractions of minimal prime ideals of $B$. Consequently $$\dim \left(B^{\prime}/m_R^t B^{\prime}\right) \leq \dim B^{\prime} -1\leq \dim B -1.$$ From the classical theory of Hilbert series, it follows that for $n>>0$, the length function $$n \mapsto l_R\left(H^0_{m_R}\left(\dfrac{C_{an+r}}{B_{an+r}}\right)\right) = l_R\left(H^0_{m_R}\left(\dfrac{B^{\prime}_n C_{a+r}}{B^{\prime}_n B_{a+r}}\right)\right)$$ is a polynomial in $n$ of degree at most $\dim B -2$. Hence $$\lim\limits_{n\to\infty} \dfrac{l_R\left(H^0_{m_R}\left(C_{an+r}/B_{an+r}\right)\right)}{(an+r)^{\dim B -1}}=0.$$ Since this is true for all $r=0,\ldots,a-1$, therefore $$\lim\limits_{n\to\infty} \dfrac{l_R\left(H^0_{m_R}\left(C_{n}/B_{n}\right)\right)}{n^{\dim B -1}}=0.$$
\end{proof}

\begin{theorem}\cite[Theorem 4]{das}\label{main2}
Suppose that $(R,m_R)$ is an excellent local ring, $$A = \bigoplus_{n\geq 0}A_n \subset B = \bigoplus_{n\geq 0}B_n$$ is a graded inclusion of $R$-algebras with $A_0 = B_0 = R$. Further assume that $B$ is a reduced standard graded Noetherian $R$-algebra and $A$ is a graded $R$-subalgebra of $B$. Suppose that if $P$ is a minimal prime ideal of $B$ (which is necessarily homogeneous) then $P\cap R \neq m_R$ and if $A_1\subset P$ then $A_n \subset P$ for all $n\geq 1$. Further suppose that there exists an integer $p\geq 1$ such that for all integers $c\geq 1$, we have
\begin{equation*}
    \limsup\limits_{n\to\infty}\dfrac{l_R\left(A_n/(m_R^{cn}B_n)\cap A_n\right)}{n^p}<\infty.
\end{equation*}
Then the limit $$\lim_{n\to \infty}\dfrac{l_R\left(A_n/(m_R^{cn}B_n)\cap A_n\right)}{n^p}$$ exists for all positive integers $c$.
\end{theorem}
We are about to prove the main theorem in this section.
\begin{theorem}\label{maintheorem}
Suppose that $(R,m_R)$ is an excellent local ring and $$A = \bigoplus_{n\in\mathbb{N}}A_n \subset B = \bigoplus_{n\in\mathbb{N}}B_n$$ is a graded inclusion of reduced finitely generated graded $R$-algebras. Suppose that if $P$ is a minimal prime ideal of $B$ (which is necessarily homogeneous) then $P\cap R \neq m_R$ and if $A_1\subset P$ then $A_n \subset P$ for all $n\geq 1$. Then the limit $$\lim\limits_{n\to\infty}\dfrac{l_R\left(H^0_{m_R}\left(B_n/A_n\right)\right)}{n^{\dim B -1}}$$ exists.
\end{theorem}
\begin{proof}
In view of Lemma \ref{reduction}, we can further assume that $B$ is also standard graded. From the proof of Theorem \ref{limsup}, it follows that there exists an integer $c>0$ such that 
\begin{equation}\label{maineq}
    l_R\left(\dfrac{A_n \colon_{B_n} m_R^{\infty}}{A_n}\right) = l_R\left(\dfrac{\left(A_n \colon_{B_n} m_R^{\infty}\right)}{(m_R^{cn}B_n) \cap \left(A_n \colon_{B_n} m_R^{\infty}\right)}\right) - l_R\left(\dfrac{A_n}{(m_R^{cn}B_n) \cap A_n}\right)
\end{equation}
for all $n\geq 1$. Let $A^{\prime} = \bigoplus_{n\in\mathbb{N}}A^{\prime}_n$ be a graded $R$-subalgebra of $B$ defined by $A^{\prime}_n = \left(A_n\colon_{B_n}m_R^{\infty}\right)$. Using Lemma \ref{bound}, we have that
\begin{align*}
    \limsup\limits_{n\to\infty}\dfrac{l_R\left(A_n/(m_R^{cn}B_n) \cap A_n\right)}{n^{\dim B -1}} &\leq \limsup\limits_{n\to\infty}\dfrac{l_R\left(B_n/m_R^{cn}B_n\right)}{n^{\dim B -1}} <\infty,\\
     \limsup\limits_{n\to\infty}\dfrac{l_R\left(\left(A_n \colon_{B_n} m_R^{\infty}\right)/(m_R^{cn}B_n) \cap \left(A_n \colon_{B_n} m_R^{\infty}\right)\right)}{n^{\dim B -1}} &\leq \limsup\limits_{n\to\infty}\dfrac{l_R\left(B_n/m_R^{cn}B_n\right)}{n^{\dim B -1}} <\infty.
\end{align*}
Let $P$ be a minimal prime ideal of $B$ (which is necessarily homogeneous) and using the hypothesis, we observe that $$\left(A_1 \colon_{B_1} m_R^{\infty}\right) \subset P \implies A_1 \subset P \implies A_n \subset P \implies \left(A_n \colon_{B_n} m_R^{\infty}\right) \subset P$$ for all $n\geq 1$. From Theorem \ref{main2}, it follows that the limits 
\begin{align*}
 &\lim\limits_{n\to\infty}\dfrac{l_R\left(A_n/(m_R^{cn}B_n) \cap A_n\right)}{n^{\dim B -1}},\\
 &\lim\limits_{n\to\infty}\dfrac{l_R\left(\left(A_n \colon_{B_n} m_R^{\infty}\right)/(m_R^{cn}B_n) \cap \left(A_n \colon_{B_n} m_R^{\infty}\right)\right)}{n^{\dim B -1}}
\end{align*}
 exist. From equation \eqref{maineq}, we can now deduce that the limit $$\lim\limits_{n\to\infty}\dfrac{l_R\left(\left(A_n \colon_{B_n} m_R^{\infty}\right)/A_n\right)}{n^{\dim B -1}}$$ exists.
\end{proof}

We recover an important special case of Cutkosky's result \cite[Theorem 6.1.]{DC6} by an application of Theorem $\ref{maintheorem}$.
\begin{corollary}\label{noeth}
Suppose that $(R,m_R)$ is an excellent reduced local ring of dim $d>0$ and $\mathcal{I} = \{I_n\}$ is a Noetherian graded family of ideals in $R$. Further suppose that if $P$ is a minimal prime ideal of $R$ and $I_1 \subset P$ then $I_n \subset P$ for all $n\geq 1$. Then the limit $$\lim\limits_{n\to\infty}\dfrac{l_R\left(H^0_{m_R}\left(R/I_n\right)\right)}{n^d}$$ exists.
\end{corollary}
\begin{proof}
Let $t$ be an indeterminate and there are graded inclusions $$A:= R[I_1t, I_2t^2,\ldots] = \bigoplus\limits_{n\in\mathbb{N}}I_nt^n \subset B:= R[t] = \bigoplus\limits_{n\in\mathbb{N}}Rt^n$$ of Noetherian graded $R$-algebras. Also there are isomorphisms $$\dfrac{\left(I_n\colon_R\,m_R^{\infty}\right)}{I_n} \cong H^0_{m_R}\left(\dfrac{R}{I_n}\right) \cong H^0_{m_R}\left(\dfrac{B_n}{A_n}\right)$$ where $B_n = Rt^n$ and $A_n = I_nt^n$. Any minimal prime ideal of $R[t]$ is of the form $PR[t]$ where $P$ is a minimal prime ideal of $R$. Let $P$ be a minimal prime ideal of $R$. Then $$PR[t]\cap R = P \neq m_R$$ as $\dim R > 0$. Moreover $\dim R[t] = d+1$ and the conclusions of the corollary now follow from Theorem $\ref{maintheorem}$.
\end{proof}

\section{Mixed epsilon multiplicities of monomial ideals}

Throughout this section we shall consider the situation where $R=K[X_1,\ldots,X_d]$ is the polynomial ring in $d$ variables over a field $K$ and $m_R$ is the graded maximal ideal of $R$. Specifically, we shall only concentrate on monomial ideals in such rings. One of the key results in this section is the multi-graded generalization of \cite[Theorem 3.8.]{DM} by Dao and Monta\~{n}o.

\begin{theorem}\label{quasigrowth}
Let $\mathcal{I}(1) = \{I(i)_n\}_{n\in\mathbb{N}}, \ldots, \mathcal{I}(r) = \{I(r)_n\}_{n\in\mathbb{N}}$ be Noetherian graded families of monomial ideals. Fix an integer $t\geq 0$ and assume that $$l_R\left(H^t_{m_R}\left(R/I(1)_{n_1}\cdots I(r)_{n_r}\right)\right)<\infty$$ for all $n_1,\ldots,n_r>>0$. Then there exist periodic functions $\sigma_{i_1,\ldots,i_r} \colon \mathbb{N}^r \to \mathbb{Q}$ for all $i_1+\cdots+i_r \leq d$ such that there is an equality $$l_R\left(H^t_{m_R}\left(\dfrac{R}{I(1)_{n_1}\cdots I(r)_{n_r}}\right)\right) = \sum\limits_{i_1+\cdots+i_r\leq d}\sigma_{i_1,\ldots,i_r}(n_1,\ldots,n_r)n_1^{i_1}\cdots n_r^{i_r}$$ for all $n_1,\ldots,n_r>>0$.
\end{theorem}
Theorem \ref{quasigrowth} has been proven by Dao and Monta\~{n}o for $r=1$ in \cite{DM}. Their proof can be suitably adapted to get this more general statement. Therefore we shall only indicate the necessary modifications. The interested reader is requested to refer to \cite[Section 3]{DM} and \cite{KW} for notations, definitions and background. 

\begin{theorem}\cite[Proposition 1.10.]{KW}\label{Presburger}
 Given a function $f \colon \mathbb{N}^d \to \mathbb{N}$ and consider the following three possible properties:
 \begin{enumerate}
  \item[$(i)$] $f$ is a Presburger counting function.
  \item[$(ii)$] $f$ is a piecewise quasi-polynomial.
  \item[$(iii)$] The generating function $\sum\limits_{(a_1,\ldots,a_d)\in\mathbb{N}^d} f(a_1,\ldots,a_d)X_1^{a_1}\cdots X_d^{a_d}$ is rational.
 \end{enumerate}
Then we have the implications $(i) \implies (ii) \iff (iii)$.
\end{theorem}
The following result generalizes \cite[Proposition 3.5.]{DM}.
\begin{proposition}\label{quasi}
 Let $\mathcal{I}(i,j) = \{I(i,j)_n\}_{n\in\mathbb{N}}$ be a Noetherian graded family of monomial ideals in $R$ for $i=1,\ldots,r$ and $j=1,\ldots,p$. Let $\mathcal{J}(i,j) = \{J(i,j)_n\}_{n\in\mathbb{N}}$ be another Noetherian graded family of monomial ideals in $R$ for $i=1,\ldots,r$ and $j=1,\ldots,q$. For every $(n_1,\ldots,n_r) \in \mathbb{N}^r$ consider the set
 \begin{equation*}
  \begin{split}
   S_{(n_1,\ldots,n_r)} = \{(a_1,\ldots,a_d)\in\mathbb{N}^d \mid &X_1^{a_1}\cdots X_d^{a_d} \in I(1,j)_{n_1}\cdots I(r,j)_{n_r} \; \forall 1\leq j\leq p, \;\text{and}\\ & X_1^{a_1}\cdots X_d^{a_d} \notin J(1,j)_{n_1}\cdots J(r,j)_{n_r} \; \forall 1\leq j\leq q \}.
  \end{split}
 \end{equation*}
 Define $f(n_1,\ldots,n_r) = \#S_{(n_1,\ldots,n_r)}$ for every $(n_1,\ldots,n_r)\in \mathbb{N}^r$ and assume that $f(n_1,\ldots,n_r)$ is finite for all $n_1,\ldots,n_r>>0$. Then there exist periodic functions $\sigma_{i_1,\ldots,i_r} \colon \mathbb{N}^r \to \mathbb{Q}$ for all $i_1+\cdots+i_r \leq d$ such that there is an equality $$f(n_1,\ldots,n_r) = \sum\limits_{i_1+\cdots+i_r\leq d}\sigma_{i_1,\ldots,i_r}(n_1,\ldots,n_r)n_1^{i_1}\cdots n_r^{i_r}$$ for all $n_1,\ldots,n_r>>0$.
\end{proposition}
\begin{proof}
For any two vectors $\mathbf{u}=(u_1,\ldots,u_r)\in \mathbb{Z}^r$ and $\mathbf{v}=(v_1,\ldots,v_r)\in \mathbb{Z}^r$, we say that $\mathbf{u} \succeq \mathbf{v}$ if $u_i\geq v_i$ for all $1\leq i\leq r$. Moreover, we say $\mathbf{u}\succ \mathbf{v}$ if $\mathbf{u}\succeq \mathbf{v}$ and $u_i > v_i$ for some $i$. Fix $\mathbf{N} = (N_1,\ldots,N_r) \in \mathbb{N}^r$ such that $f(n_1,\ldots,n_r)<\infty$ for every $$\mathbf{n}:= (n_1,\ldots,n_r)\succeq \mathbf{N}$$ and such that the Noetherian $\mathbb{N}^r$-graded $R$-algebras $$\bigoplus_{(h_1,\ldots,h_r)\in\mathbb{N}^r}I(1,j)_{h_1}\cdots I(r,j)_{h_r}t_1^{h_1}\cdots t_r^{h_r}$$ for all $j=1,\ldots,p$ and $$\bigoplus_{(h_1,\ldots,h_r)\in\mathbb{N}^r}J(1,j)_{h_1}\cdots J(r,j)_{h_r}t_1^{h_1}\cdots t_r^{h_r}$$ for all $j=1,\ldots,q$ are generated in degrees $\preceq \mathbf{N}$. For a vector $\mathbf{u}=(u_1,\ldots,u_d) \in \mathbb{N}^d$, we write $\mathbf{X}^{\mathbf{u}} := X_1^{u_1}\cdots X_d^{u_d}$. Suppose that $$\left\{\mathbf{X}^{\mathbf{u}_{j,\mathbf{h},1}}, \ldots, \mathbf{X}^{\mathbf{u}_{j,\mathbf{h},\lambda_{j,\mathbf{h}}}}\right\}$$ is the minimal monomial generating set of $I(1,j)_{h_1}\cdots I(r,j)_{h_r}$ for $j = 1,\ldots,p$ and $\mathbf{0} \prec \mathbf{h}:=(h_1,\ldots,h_r) \preceq \mathbf{N}$. Fix $\mathbf{a} = (a_1,\ldots,a_d) \in \mathbb{N}^d$ and consider the Presburger formula
\begin{equation*}
  Y_j(\mathbf{a};\mathbf{n}) = \left(\exists t_{\mathbf{h},s}\in \mathbb{N},\mathbf{0}\prec \mathbf{h} \preceq \mathbf{N}, 1\leq s \leq \lambda_{j,\mathbf{h}}\right)\left(\left(\sum\limits_{\mathbf{h},s}t_{\mathbf{h},s}\mathbf{h} = \mathbf{n}\right) \land \left(\mathbf{a}\succeq \sum\limits_{\mathbf{h},s}t_{\mathbf{h},s}\mathbf{u}_{j,\mathbf{h},s}\right)\right).
\end{equation*}
Similarly let $$\left\{\mathbf{X}^{\mathbf{v}_{j,\mathbf{h},1}}, \ldots, \mathbf{X}^{\mathbf{v}_{j,\mathbf{h},\eta_{j,\mathbf{h}}}}\right\}$$ be the minimal monomial generating set of $J(1,j)_{h_1}\cdots J(r,j)_{h_r}$ for $j = 1,\ldots,q$ and $\mathbf{0} \prec \mathbf{h}:=(h_1,\ldots,h_r) \preceq \mathbf{N}$ and consider the Presburger formula $$Z_j(\mathbf{a};\mathbf{n}) = \left(\forall t_{\mathbf{h},s}\in \mathbb{N},\mathbf{0}\prec \mathbf{h} \preceq \mathbf{N}, 1\leq s \leq \eta_{j,\mathbf{h}}\right)\left(\left(\sum\limits_{\mathbf{h},s}t_{\mathbf{h},s}\mathbf{h} \prec \mathbf{n}\right) \lor \neg \left(\mathbf{a}\succeq \sum\limits_{\mathbf{h},s}t_{\mathbf{h},s}\mathbf{v}_{j,\mathbf{h},s}\right)\right).$$ It follows that for all possible $j$ and $\mathbf{n}$, $\mathbf{X}^{\mathbf{a}} \in I(1,j)_{n_1}\cdots I(r,j)_{n_r}$ if and only if $Y_j(\mathbf{a};\mathbf{n})$ is satisfied, and $\mathbf{X}^{\mathbf{a}} \notin J(1,j)_{n_1}\cdots J(r,j)_{n_r}$ if and only if $Z_j(\mathbf{a};\mathbf{n})$ is satisfied. Therefore $$g(\mathbf{n}):= f(\mathbf{n}+\mathbf{N}) = \#\left\{\mathbf{a} \in \mathbb{N}^d \mid \left(\bigwedge\limits_{j=1}^p Y_j(\mathbf{a};\mathbf{n}+\mathbf{N})\right)\land \left(\bigwedge\limits_{j=1}^q Z_j(\mathbf{a};\mathbf{n}+\mathbf{N})\right)\right\}$$ is a Presburger counting function, which is a piecewise quasi-polynomial from Theorem \ref{Presburger}. This implies that that $g(n_1,\ldots,n_r)$, and hence $f(n_1,\ldots,n_r)$, agrees with a quasi-polynomial for all $n_i>>0$. An analysis of the proof of Theorem \ref{Presburger} shows that this quasi-polynomial has total degree at most $d$.
\end{proof}
The following result generalizes \cite[Lemma 3.7.]{DM}.
\begin{lemma}\label{quasi2}
 Let $\mathcal{I}(1) = \{I(i)_n\}_{n\in\mathbb{N}}, \ldots, \mathcal{I}(r) = \{I(r)_n\}_{n\in\mathbb{N}}$ be Noetherian graded families of monomial ideals. Fix a simplicial complex $\Delta^{\prime}$ with support $\{1,\ldots,d\}$. Consider the function $$f_{\Delta^{\prime}}(n_1,\ldots,n_r) = \#\{\mathbf{a}\in\mathbb{N}^d \mid \Delta_{\mathbf{a}}\left(I(1)_{n_1}\cdots I(r)_{n_r}\right)=\Delta^{\prime}\},$$ and assume that $f_{\Delta^{\prime}}(n_1,\ldots,n_r)<\infty$ for all $n_1,\ldots,n_r>>0$. Then there exist periodic functions $\sigma_{i_1,\ldots,i_r} \colon \mathbb{N}^r \to \mathbb{Q}$ for all $i_1+\cdots+i_r \leq d$ such that there is an equality $$f_{\Delta^{\prime}}(n_1,\ldots,n_r) = \sum\limits_{i_1+\cdots+i_r\leq d}\sigma_{i_1,\ldots,i_r}(n_1,\ldots,n_r)n_1^{i_1}\cdots n_r^{i_r}$$ for all $n_1,\ldots,n_r>>0$.
\end{lemma}
\begin{proof}
For any subset $F$ of $\{1,\ldots,d\}$, it is not hard to verify that if $\mathcal{I}=\{I_n\}_{n\in\mathbb{N}}$ is a Noetherian graded family of ideals then $\mathcal{I}_F:=\{(I_n)_F\}_{n\in\mathbb{N}}$ is also a Noetherian graded family of ideals. Moreover, for any two ideals $I$ and $J$, we have $(IJ)_F = I_F J_F$. From the description of $\Delta_{\mathbf{a}}\left(I(1)_{n_1}\cdots I(r)_{n_r}\right)$, it follows that
 \begin{equation*}
  \begin{split}
   f(n_1,\ldots,n_r) = \#\{\mathbf{a}\in\mathbb{N}^d \mid &\mathbf{X}^{\mathbf{a}} \in (I(1)_{n_1})_F\cdots (I(r)_{n_r})_F \; \forall F\in \Delta^{\prime}, \;\text{and}\\ &\mathbf{X}^{\mathbf{a}} \notin (I(1)_{n_1})_F\cdots (I(r)_{n_r})_F \; \forall F\notin \Delta^{\prime}\}.
  \end{split}
 \end{equation*}
The conclusions of the lemma is now a consequence of Proposition \ref{quasi}.
\end{proof}
We are now ready to prove one of our main results.
\begin{proof}[Proof of Theorem \ref{quasigrowth}]
Consider $n_1,\ldots,n_r>>0$ and fix $\mathbf{a}=(a_1,\ldots,a_d)\in\mathbb{Z}^d$ such that $H^t_{m_R}\left(R/I(1)_{n_1}\cdots I(r)_{n_r}\right)_{\mathbf{a}} \neq 0$. Since $l_R\left(H^t_{m_R}\left(R/I(1)_{n_1}\cdots I(r)_{n_r}\right)\right)<\infty$, it follows from \cite[Proposition 1]{TAK} that $\mathbf{a}\in\mathbb{N}^d$. Since $I(j)_1^{n_j} \subset I(j)_{n_j}$ for all $j=1,\ldots,r$, so $\sqrt{I(1)_1\cdots I(r)_1} \subset \sqrt{I(1)_{n_1}\cdots I(r)_{n_r}}$. Therefore $\Delta(I(1)_{n_1}\cdots I(r)_{n_r})$ is a subcomplex of $\Delta(I(1)_1\cdots I(r)_1)$, and then so is $\Delta_{\mathbf{a}}\left(I(1)_{n_1}\cdots I(r)_{n_r}\right)$. Now,
\begin{align*}
 &l_R\left(H^t_{m_R}\left(R/I(1)_{n_1}\cdots I(r)_{n_r}\right)\right)\\
 &= \sum\limits_{\mathbf{a}\in\mathbb{N}^d}\dim_K H^t_{m_R}\left(R/I(1)_{n_1}\cdots I(r)_{n_r}\right)_{\mathbf{a}}\\
 &= \sum\limits_{\mathbf{a}\in\mathbb{N}^d}\dim_K \tilde{H}_{t-1}\left(\Delta_{\mathbf{a}}(I(1)_{n_1}\cdots I(r)_{n_r}),K\right)\\
 &= \sum\limits_{\Delta^{\prime}\subset \Delta(I(1)_1\cdots I(r)_1)} \sum\limits_{\{\mathbf{a}\in\mathbb{N}^d \mid \Delta_{\mathbf{a}}\left(I(1)_{n_1}\cdots I(r)_{n_r}\right)=\Delta^{\prime}\}} \dim_K \tilde{H}^{t-1}(\Delta^{\prime},K)\\
 &= \sum\limits_{\Delta^{\prime}\subset \Delta(I(1)_1\cdots I(r)_1)} \dim_K \tilde{H}^{t-1}(\Delta^{\prime},K)f_{\Delta^{\prime}}(n_1,\ldots,n_r),
\end{align*}
where $f_{\Delta^{\prime}}(n_1,\ldots,n_r)$ is defined as in Lemma \ref{quasi2}. The conclusion now follows because the function $f_{\Delta^{\prime}}(n_1,\ldots,n_r)$ is a quasi-polynomial by Lemma \ref{quasi2} and a finite linear combination of quasi-polynomials is a quasi-polynomial. 
\end{proof}

The following result is useful for certain approximation purposes.

\begin{lemma}\label{genjmult}
 Suppose that $(R,m_R)$ is a Noetherian local ring of Krull dimension $d>0$ and $I_1,\ldots,I_r$ be ideals in $R$. Fix non-negative integers $a_1,\ldots,a_r$. Then there exist constants $\sigma_{i_1,\ldots,i_r} \in \mathbb{Q}$ for all $i_1+\cdots+i_r\leq d-1$ such that there is an equality $$l_R\left(H^0_{m_R}\left(\dfrac{I_1^{n_1}\cdots I_r^{n_r}}{I_1^{n_1+a_1}\cdots I_r^{n_r+a_r}}\right)\right) = \sum\limits_{i_1+\cdots+i_r\leq d-1}\sigma_{i_1,\ldots,i_r}n_1^{i_1}\cdots n_r^{i_r}$$ for all $n_1,\ldots,n_r>>0$.
\end{lemma}
\begin{proof}
 Let $$A = \bigoplus\limits_{(n_1,\ldots,n_r)\in\mathbb{N}^r}\dfrac{I_1^{n_1}\cdots I_r^{n_r}}{I_1^{n_1+a_1}\cdots I_r^{n_r+a_r}},$$ which is a standard $\mathbb{N}^r$-graded Noetherian $R/I_1^{a_1}\ldots I_r^{a_r}$-algebra with Krull dimension at most $d+r-1$. Then $$H^0_{m_RA}(A) = \bigoplus\limits_{(n_1,\ldots,n_r)\in\mathbb{N}^r}H^0_{m_R}\left(\dfrac{I_1^{n_1}\cdots I_r^{n_r}}{I_1^{n_1+a_1}\cdots I_r^{n_r+a_r}}\right)$$ is a homogeneous ideal of $A$ which is annihilated by $m_R^cA$ for some $c>0$. Therefore $H^0_{m_RA}(A)$ can be regarded as a finitely generated $\mathbb{N}^r$-graded module over $$\dfrac{A}{m_R^cA} = \bigoplus\limits_{(n_1,\ldots,n_r)\in\mathbb{N}^r}\dfrac{I_1^{n_1}\cdots I_r^{n_r}}{m_R^cI_1^{n_1}\cdots I_r^{n_r} + I_1^{n_1+a_1}\cdots I_r^{n_r+a_r}},$$ which is a standard $\mathbb{N}^r$-graded Noetherian algebra over the Artinian local ring $\dfrac{R}{m_R^c+I_1^{a_1}\ldots I_r^{a_r} }$. Using \cite[Lemma 1.1.]{HHRT}, we find that $$\dim \mathrm{Supp}_{++}\left(H^0_{m_RA}(A)\right) \leq \dim \mathrm{Proj} \left(\dfrac{A}{m_R^cA}\right) \leq \dim \dfrac{A}{m_R^cA} - r \leq d-1.$$ The conclusions of the lemma now follow from \cite[Theorem 4.1.]{HHRT}.
\end{proof}

\begin{remark}
If $I\subset R$ is a monomial ideal then the analytic spread of $I$ is one more than the maximal dimension of a bounded facet of its Newton polyhedron (see \cite[Theorem 2.3.]{BA}).
\end{remark}

The next result is crucial to establish the notion of mixed epsilon multiplicity for monomial ideals.

\begin{theorem}\label{mainmain}
Suppose that $I_1,\ldots,I_r$ are monomial ideals in $R$ and assume that the analytic spread of $I_1\cdots I_r$ is $d$. Then there exist periodic functions $\sigma_{i_1,\ldots,i_r} \colon \mathbb{N}^r \to \mathbb{Q}$ for all $i_1+\cdots +i_r\leq d$ such that there is an equality $$l_R\left(H^0_{m_R}\left(\dfrac{R}{I_1^{n_1}\cdots I_r^{n_r}}\right)\right) = \sum\limits_{i_1+\cdots+i_r\leq d} \sigma_{i_1,\ldots,i_r}(n_1,\ldots,n_r)n_1^{i_1}\cdots n_r^{i_r}$$ for all $n_1,\ldots,n_r>>0$. Moreover, $\sigma_{i_1,\ldots,i_r}$ is a constant function whenever $i_1+\cdots +i_r = d$ and $\sigma_{i_1,\ldots,i_r}\neq 0$ for some $i_1+\cdots +i_r = d$.
\end{theorem}
\begin{proof}
 From the description of the analytic spread of monomial ideals \cite[Theorem 2.3.]{BA}, we observe that if the analytic spread of $I_1\cdots I_r$ is maximal then the analytic spread of $I_1^{n_1}\cdots I_r^{n_r}$ is also maximal for all $(n_1,\ldots,n_r)\in\mathbb{Z}_{\geq 1}^d$. It follows from \cite[Theorem 4.4]{BJ2} that $$\lim\limits_{t\to\infty}\dfrac{l_R\left(H^0_{m_R}\left(R/I_1^{tn_1}\cdots I_r^{tn_r}\right)\right)}{t^d} \neq 0$$ for all $(n_1,\ldots,n_r)\in\mathbb{Z}_{\geq 1}^d$. 
 Now from Theorem \ref{quasigrowth}, we know that there exist periodic functions $\sigma_{i_1,\ldots,i_r} \colon \mathbb{N}^r \to \mathbb{Q}$ for all $i_1+\cdots +i_r\leq d$ such that there is an equality 
 \begin{equation}\label{quasip}
  l_R\left(H^0_{m_R}\left(\dfrac{R}{I_1^{n_1}\cdots I_r^{n_r}}\right)\right) = \sum\limits_{i_1+\cdots+i_r\leq d} \sigma_{i_1,\ldots,i_r}(n_1,\ldots,n_r)n_1^{i_1}\cdots n_r^{i_r}
 \end{equation}
for all $n_1,\ldots,n_r>>0$. Further check that
\begin{equation}\label{quasip2}
 0 \neq \lim\limits_{t\to\infty}\dfrac{l_R\left(H^0_{m_R}\left(R/I_1^{tn_1}\cdots I_r^{tn_r}\right)\right)}{t^d} = \sum\limits_{i_1+\cdots+i_r = d} \sigma_{i_1,\ldots,i_r}(n_1,\ldots,n_r)n_1^{i_1}\cdots n_r^{i_r}
\end{equation}
 for all $n_1,\ldots,n_r>>0$. We shall denote the numerical quasi-polynomial in \eqref{quasip} by $Q(n_1,\ldots,n_r)$ and let $a$ be the common period of the coefficients $\sigma_{i_1,\ldots,i_r}(n_1,\ldots,n_r)$ of $Q(n_1,\ldots,n_r)$. Suppose that $b_1,\ldots,b_r \in \mathbb{N}$ with $0\leq b_i <a$ for all $i=1,\ldots,r$. Then
 \begin{align*}
  &Q(an_1+b_1,\ldots,an_r+b_r)\\
  &= \sum\limits_{i_1+\cdots+i_r \leq d}\sigma_{i_1,\ldots,i_r}(an_1+b_1,\ldots,an_r+b_r)(an_1+b_1)^{i_1}\cdots (an_r+b_r)^{i_r}\\
  &= \sum\limits_{i_1+\cdots+i_r = d}\sigma_{i_1,\ldots,i_r}(b_1,\ldots,b_r)a^d n_1^{i_1}\cdots n_r^{i_r} + (\text{lower order terms in $n_1,\ldots,n_r$})
 \end{align*}
is a polynomial in $n_1,\ldots,n_r$ of total degree at most $d$. By using the subadditivity of $l_R\left(H^0_{m_R}(\_)\right)$, we conclude that for all $n_1,\ldots,n_r>>0$, 
 \begin{align}
  Q(an_1+b_1,\ldots,an_r+b_r) &\leq Q(an_1,\ldots,an_r) + l_R\left(H^0_{m_R}\left(\dfrac{I_1^{an_1}\cdots I_r^{an_r}}{I_1^{an_1+b_1}\cdots I_r^{an_r+b_r}}\right)\right),\label{growth1}\\
  Q(a(n_1+1),\ldots,a(n_r+1)) &\leq Q(an_1+b_1,\ldots,an_r+b_r) + l_R\left(H^0_{m_R}\left(\dfrac{I_1^{an_1+b_1}\cdots I_r^{an_r+b_r}}{I_1^{a(n_1+1)}\cdots I_r^{a(n_r+1)}}\right)\right).\label{growth2}
 \end{align}
It follows from Lemma \ref{genjmult} that there exist numerical polynomials $\xi_1(n_1,\ldots,n_r)$ and $\xi_2(n_1,\ldots,n_r)$ of total degrees at most $d-1$ such that
\begin{align*}
 \xi_1(n_1,\ldots,n_r) &= l_R\left(H^0_{m_R}\left(\dfrac{I_1^{an_1}\cdots I_r^{an_r}}{I_1^{an_1+b_1}\cdots I_r^{an_r+b_r}}\right)\right),\\
 \xi_2(n_1,\ldots,n_r) &= l_R\left(H^0_{m_R}\left(\dfrac{I_1^{an_1+b_1}\cdots I_r^{an_r+b_r}}{I_1^{a(n_1+1)}\cdots I_r^{a(n_r+1)}}\right)\right)
\end{align*}
for all $n_1,\ldots,n_r>>0$. The inequalities \eqref{growth1} and \eqref{growth2} now imply that $$\sigma_{i_1,\ldots,i_r}(0,\ldots,0)  = \sigma_{i_1,\ldots,i_r}(b_1,\ldots,b_r)$$ whenever $i_1+\ldots+i_r=d$. Combining \eqref{quasip2} and the above equality, proves our theorem.
\end{proof}

\section{Limits of graded family of monomial ideals}
The following theorem is a special case of a result of Cutkosky. 

\begin{theorem}\cite[Theorem 6.1.]{DC6}\label{special}
Let $R=K[X_1,\ldots,X_d]$ be the polynomial ring in $d$ variables over a field $K$, $m_R = (X_1,\ldots,X_d)$ be the graded maximal ideal of $R$ and $\mathcal{I}=\{I_n\}_{n\in\mathbb{N}}$ be a graded family of monomial ideals in $R$. Suppose that there exists an integer $c>0$ such that 
\begin{equation}\label{linearity}
I_n \cap m_R^{cn} = \left(I_n \colon_R m_R^{\infty}\right) \cap m_R^{cn}
\end{equation}
for all $n\geq 0$. Then the limit $$\lim_{n\to\infty}\dfrac{l_R\left(H^0_{m_R}\left(R/I_n\right)\right)}{n^d}$$ exists.
\end{theorem}

Any Noetherian graded family of monomial ideals will satisfy the \emph{linear growth condition} \eqref{linearity} of Theorem \ref{special}. The subsequent simple example (due to Cutkosky) shows that it is possible to construct non-Noetherian graded family of monomial ideals which satisfy the linear growth condition \eqref{linearity}.

\begin{example}
Let $R=K[X]$ be the polynomial ring in single variable over a field $K$ and let $m_R = (X)$ be the graded maximal ideal of $R$. Fix an irrational number $\alpha >0$. Define
\begin{align*}
I_0 &= R,\\
I_n &= \left(X^{\left\lceil n\alpha \right\rceil}\right) \quad \forall n\geq 1.
\end{align*}
Then $\mathcal{I} = \{I_n\}_{n\in\mathbb{N}}$ is a filtration of monomial ideals in $R$ and $$\lim_{n\to\infty}\dfrac{l_R\left(R/I_n\right)}{n} = \alpha.$$
\end{example}

We now produce a simple example showing that the sequence $\left\{l_R\left(H^0_{m_R}\left(R/I_n\right)\right)\right\}_{n\in\mathbb{N}}$ can be quite arbitrary for general graded families of monomial ideals $\mathcal{I}= \{I_n\}_{n\in\mathbb{N}}$. 

\begin{example}\label{counter}
Let $R=K[X,Y]$ be the polynomial ring in two variables over a field $K$ and let $m_R = (X,Y)$ be the graded maximal ideal of $R$. Let $\{a_n\}_{n\in\mathbb{N}}$ be any sequence of natural numbers. Define
\begin{align*}
 I_0 &= R,\\
 I_n &= \left(XY^{a_n}, X^2\right) \quad \forall n\geq 1.
\end{align*}
Then $\mathcal{I} = \{I_n\}_{n\in\mathbb{N}}$ is a graded family of monomial ideals in $R$ and $$l_R\left(H^0_{m_R}\left(R/I_n\right)\right) = a_n \quad \forall n\geq 1.$$
\end{example}

\begin{proof}
In order to show that $\mathcal{I} = \{I_n\}_{n\in\mathbb{N}}$ is a graded family, it suffices to show that $I_nI_m \subset I_{n+m}$ for all $n,m\in \mathbb{N}$. Notice that $$ I_nI_m \subset \left(X\right).\left(X\right) = \left(X^2\right) \subset I_{n+m}.$$ Also observe that $(I_n \colon_{R} m_R^{\infty}) = (X)$ for all $n\geq 1$. Therefore
\begin{align*}
l_R\left(H^0_{m_R}\left(R/I_n\right)\right) &= l_R\left((X)/I_n\right) = \#\left\{(a,b)\in\mathbb{Z}^2 \mid X^aY^b \in (X)\setminus I_n\right\}\\
&= a_n.
\end{align*}
\end{proof}

The following result can be seen as an extension of Theorem $\ref{special}$ in two variables.

\begin{theorem}\label{mainbound}
 Let $R=K[X,Y]$ be the polynomial ring in two variables over a field $K$, $m_R = (X,Y)$ be the graded maximal ideal of $R$ and $\mathcal{I}=\{I_{n}\}_{n\in\mathbb{N}}$ be a graded family of monomial ideals in $R$. Suppose that there exists a positive integer constant $c$ such that $$I_{n}\cap m_R^{cn^2} = \left(I_{n} \colon_R m_R^{\infty}\right)\cap m_R^{cn^2}$$ for all $n\geq 1$. Then $$\limsup_{n\to\infty} \dfrac{l_R\left(H^0_{m_R}\left(R/I_{n}\right)\right)}{n^{3}} <\infty.$$
\end{theorem}

\begin{proof}
 Since $\mathcal{I}=\{I_{n}\}_{n\in\mathbb{N}}$ is a graded family of monomial ideals in $K[X,Y]$, we have $$\left(I_{n} \colon_R m_R^{\infty}\right) = \left(X^{a_{n}}Y^{b_{n}}\right)$$ where $\{a_{n}\}_{n\in\mathbb{N}}$ and $\{b_{n}\}_{n\in\mathbb{N}}$ are sequences of non-negative integers satisfying
 \begin{align*}
  a_0 = b_0 &= 0,\\
  a_{n} + a_{m} &\geq a_{n+m},\\
  b_{n} + b_{m} &\geq b_{n+m}
 \end{align*}
for all $n\geq 1$ and $m\geq 1$. From these inequalities, it follows that $a_{n} \leq a_{1}n$ and $b_{n} \leq b_{1}n$ for all $n\geq 1$. By increasing $c$ if necessary, we may assume that $c \geq \max\{a_{1},b_{1}\}$. Define the ideals $$J_{n}=\left(X^{a_{n}}Y^{cn^2},X^{a_{1}n}Y^{cn},X^{cn}Y^{b_{1}n},X^{cn^2}Y^{b_{n}}\right)$$ for all $n\geq 1$. Now observe that
\begin{align*}
 J_{n} \subset \left(I_{n}\colon_R m_R^{\infty}\right)\cap m_R^{cn^2} + I_{1}^n &\subset I_{n},\\
 J_{n} \colon_R m_R^{\infty} = \left(X^{a_{n}}Y^{b_{n}}\right) &= I_{n} \colon_R m_R^{\infty}
\end{align*}
for all $n\geq 1$. Therefore 
\begin{align*}
 l_R\left(H^0_{m_R}\left(R/I_{n}\right)\right) &\leq l_R\left(H^0_{m_R}\left(R/J_{n}\right)\right)\\
 &= (cn - a_{n})(cn - b_{n}) + (a_{1}n-a_{n})(cn^2 - cn)\\
 &\qquad + (b_{1}n - b_{n})(cn^2 - cn)\\
 &\leq (a_{1}+b_{1})cn^{3} + c^2n^2
\end{align*}
for all $n\geq 1$. The conclusions of the theorem follow from the above inequality.
\end{proof}

However we suspect that the denominator in this theorem can be improved. The subsequent example seems to suggest a possibility to sharpen it.

\begin{example}\label{limit}
Let $R=K[X,Y]$ be the polynomial ring in two variables over a field $K$ and let $m_R = (X,Y)$ be the graded maximal ideal of $R$. Consider the following family of ideals in $R$ recursively defined by
\begin{align*}
    I_0 &= R,\\
    I_1 &= \left(XY\right),\\
    I_n &= \left(XY^{n^2},X^{n^2}Y\right) + \sum\limits_{t=1}^{n-1}I_t I_{n-t} \quad \forall n\geq 2.
\end{align*}
Then $\mathcal{I} = \{I_n\}_{n\in\mathbb{N}}$ is a filtration of $R$ by monomial ideals and $$\lim\limits_{n\to\infty}\dfrac{l_R\left(H^0_{m_R}\left(R/I_n\right)\right)}{n^2\ln n} = 2.$$
\end{example}

This example was inspired by computations in Macaulay2. Before proving this example, we require a result.

\begin{proposition}\label{approx}
Let $R=K[X,Y]$ be the polynomial ring in two variables over a field $K$. Consider the following family of ideals in $R$ defined by
\begin{align*}
    J_0 &= R,\\
    J_n &= \left\langle X^aY^b \mid ab\geq n^2, b\geq n \right\rangle \; \forall n\geq 1,\\
    J^{\prime}_0 &= R,\\
    J^{\prime}_n &= \left\langle X^aY^b \mid ab\geq n^2, a\geq n \right\rangle \; \forall n\geq 1.
\end{align*}
Then the following statements are true.
\begin{enumerate}
\item[$(i)$] $\mathcal{J} = \{J_n\}_{n\in\mathbb{N}}$ is a filtration of $R$ by monomial ideals.
\item[$(ii)$] $YJ_n \subseteq (XY^{n^2}) + \sum\limits_{t=1}^{n-1}J_tJ_{n-t} \subseteq J_n$, $\forall n\geq 2$.
\item[$(iii)$] $\mathcal{J}^{\prime} = \{J^{\prime}_n\}_{n\in\mathbb{N}}$ is a filtration of $R$ by monomial ideals.
\item[$(iv)$] $XJ^{\prime}_n \subseteq (X^{n^2}Y) + \sum\limits_{t=1}^{n-1}J^{\prime}_tJ^{\prime}_{n-t} \subseteq J^{\prime}_n$, $\forall n\geq 2$.
\end{enumerate}
\end{proposition}

\begin{proof}
We shall only provide proofs of the parts $(i)$ and $(ii)$ because the proof of the parts $(iii)$ and $(iv)$ are analogous to the proof of the parts $(i)$ and $(ii)$ respectively.

\noindent
$(i)$ It is an easy exercise to verify that $J_n \supset J_{n+1}$ for all $n\geq 0$. Suppose that $X^{a_1}Y^{b_1} \in J_n$ and $X^{a_2}Y^{b_2} \in J_m$. Then
    \begin{align*}
        (a_1+a_2)(b_1+b_2) &= a_1b_1 + a_2b_2 + a_1b_2 + a_2b_1\\
        &\geq a_1b_1 + a_2b_2 + 2\sqrt{a_1b_1a_2b_2} \quad\text{(AM-GM inequality)}\\
        &\geq n^2 + m^2 + 2\sqrt{n^2m^2}\\
        &= (n+m)^2
    \end{align*}
 and $b_1+b_2 \geq (n+m)$. Hence $X^{a_1}Y^{b_1}X^{a_2}Y^{b_2} = X^{a_1+a_2}Y^{b_1+b_2} \in J_{n+m}$. 
 
\vspace{0.15cm} 
\noindent
$(ii)$ As $XY^{n^2}\in J_n$ and $\mathcal{J} = \{J_n\}_{n\in\mathbb{N}}$ is a graded family, it follows that $$(XY^{n^2}) + \sum\limits_{t=1}^{n-1}J_tJ_{n-t} \subseteq J_n.$$ Note that $J_n$ can be explicitly described as $$J_n = \left\langle X^aY^{\left\lceil \dfrac{n^2}{a}\right\rceil} \mid 1\leq a \leq n \right\rangle ,$$ where $\lceil x \rceil$ denotes the smallest integer greater than or equal to $x$. It now suffices to show that $$X^aY^{\left\lceil \dfrac{n^2}{a}\right\rceil + 1} \in (XY^{n^2}) + \sum\limits_{t=1}^{n-1}J_tJ_{n-t}$$ for all $1\leq a \leq n$. If $a=1$ then $XY^{n^2 +1} \in (XY^{n^2}).$ Hence assume that $2\leq a \leq n$. By means of Euclidean division, $n$ can be uniquely written as $n = qa + r$, where $q\geq 1$ and $0\leq r\leq a-1$. There exist expressions 
           \begin{align*}
             \left\lceil \dfrac{n^2}{a}\right\rceil &= \dfrac{n^2}{a} + \delta_1,\\
              \left\lceil \dfrac{(n-q)^2}{a-1}\right\rceil &= \dfrac{(n-q)^2}{a-1} + \delta_2,
           \end{align*}
          where $0\leq \delta_i < 1$ for $i=1,2$. We note that
        \begin{align*}
           \left\lceil \dfrac{n^2}{a}\right\rceil + 1 &= \dfrac{n^2}{a} + 1 + \delta_1\\
           &= \dfrac{(n-q)^2}{a-1} + q^2 - \dfrac{(n-aq)^2}{a(a-1)} + (1+\delta_1)\\
           &= \dfrac{(n-q)^2}{a-1} + q^2 + \left(1 + \delta_1 - \dfrac{r^2}{a(a-1)}\right)\\
           &= \left\lceil\dfrac{(n-q)^2}{a-1}\right\rceil + q^2 + \left(1 + \delta_1 - \delta_2 - \dfrac{r^2}{a(a-1)}\right)\\
           &\geq \left\lceil\dfrac{(n-q)^2}{a-1}\right\rceil + q^2.
        \end{align*}
    In the last step, we use the fact that the left hand side is an integer and the expression on the right hand side lying inside the parenthesis, is strictly greater than $-1$. Therefore $$X^aY^{\left\lceil\dfrac{n^2}{a}\right\rceil +1} \in \Bigg(X^aY^{\left\lceil\dfrac{(n-q)^2}{a-1}\right\rceil + q^2}\Bigg) = \Bigg(X^{a-1}Y^{\left\lceil\dfrac{(n-q)^2}{a-1}\right\rceil}\Bigg)(XY^{q^2}) \subset J_{n-q}J_q.$$
\end{proof}

\begin{proof}[Proof of example \ref{limit}]
A straightforward argument will show that $\mathcal{I} = \{I_n\}_{n\in\mathbb{N}}$ is a filtration of $R$ by monomial ideals. We now recall the family of monomial ideals $\mathcal{J} = \{J_n\}_{n\in\mathbb{N}}$ and $\mathcal{J}^{\prime} = \{J^{\prime}_n\}_{n\in\mathbb{N}}$ in $R$, as defined in Proposition \ref{approx}. We claim that 
\begin{equation}\label{bounds}
\left(Y^{n-1}J_n + X^{n-1}J^{\prime}_n\right) \subset I_n \subset \left(J_n + J^{\prime}_n\right)
\end{equation}
for all $n\geq 1$. We shall first establish that $Y^{n-1}J_n \subset I_n$ for all $n\geq 1$. By means induction hypothesis, we may assume that $Y^{t-1}J_t \subset I_t$ for all $1\leq t\leq n$. Then from Proposition \ref{approx}$(ii)$, we get $$ YJ_{n+1} \subset \left(XY^{(n+1)^2}\right) + \sum\limits_{t=1}^n J_tJ_{n+1-t},$$ which implies
   \begin{align*}
    Y^nJ_{n+1} &\subset Y^{n-1}\left(XY^{(n+1)^2}\right) + \sum\limits_{t=1}^n \left(Y^{t-1}J_t\right)\left(Y^{n-t}J_{n+1-t}\right)\\
     &\subset \left(XY^{(n+1)^2}\right) + \sum\limits_{t=1}^n \left(Y^{t-1}J_t\right) \left(Y^{n-t}J_{n+1-t}\right)\\
     &\subset \left(XY^{(n+1)^2}\right) + \sum\limits_{t=1}^n I_t I_{n+1-t} \quad \text{(induction hypothesis)}\\
     &\subset I_{n+1}.
   \end{align*}
 By a similar argument, it can be shown that $X^{n-1}J^{\prime}_n \subset I_n$ for all $n\geq 1$. Combining, we get that $\left(Y^{n-1}J_n + X^{n-1}J^{\prime}_n\right) \subset I_n$ for all $n\geq 1$. Using arguments similar to the proof of proposition \ref{approx}$(i)$, it can be shown that $$\mathcal{J} + \mathcal{J^{\prime}} := \{J_n + J^{\prime}_n\}_{n\in\mathbb{N}} = \left\{\left\langle X^aY^b \mid ab\geq n^2 \right\rangle\right\}_{n\in\mathbb{N}}$$ is a filtration of $R$ by monomial ideals. A straightforward inductive proof will show that $I_n \subset \left(J_n + J^{\prime}_n\right)$ for all $n\geq 0$. The succeeding equalities can be verified.
\begin{align*}
  &Y^{n-1}J_n + X^{n-1}J^{\prime}_n = \left\langle X^aY^{\left\lceil \dfrac{n^2}{a}\right\rceil + n-1} \mid 1\leq a \leq n \right\rangle + \left\langle X^{\left\lceil \dfrac{n^2}{a}\right\rceil + n-1}Y^a \mid 1\leq a \leq n \right\rangle,\\
  &J_n + J^{\prime}_n = \left\langle X^aY^{\left\lceil \dfrac{n^2}{a}\right\rceil} \mid 1\leq a \leq n^2 \right\rangle,\\
   &\left(Y^{n-1}J_n + X^{n-1}J^{\prime}_n\right) \colon_R m_R^{\infty} = \left(J_n + J^{\prime}_n\right) \colon_R m_R^{\infty} = I_n \colon_R m_R^{\infty} = \left(XY\right).
\end{align*}
 Therefore 
\begin{align*}
  &l_R\left(H^0_{m_R}\left(\dfrac{R}{Y^{n-1}J_n + X^{n-1}J^{\prime}_n}\right)\right) = 2\left(\sum\limits_{a=1}^n \left\lceil\dfrac{n^2}{a}\right\rceil\right) + 2\left(n^2-3n+1\right),\\
  &l_R\left(H^0_{m_R}\left(\dfrac{R}{J_n + J^{\prime}_n}\right)\right) = 2\left(\sum\limits_{a=1}^n \left\lceil\dfrac{n^2}{a}\right\rceil\right) - (n^2+2n-1).
\end{align*} 
 From \eqref{bounds}, it now follows that $$2\left(\sum\limits_{a=1}^n \left\lceil\dfrac{n^2}{a}\right\rceil\right) - (n^2+2n-1) \leq l_R\left(H^0_{m_R}\left(\dfrac{R}{I_n}\right)\right) \leq 2\left(\sum\limits_{a=1}^n \left\lceil\dfrac{n^2}{a}\right\rceil\right) + 2\left(n^2-3n+1\right).$$ Using elementary calculus, it can be shown that $$\lim_{n\to\infty}\dfrac{\sum\limits_{a=1}^n \left\lceil\dfrac{n^2}{a}\right\rceil}{n^2\ln{n}} = 1.$$ The conclusions of the example now follow from the squeeze theorem in calculus.
\end{proof}

\begin{corollary}
Let $R=K[X,Y]$ be the polynomial ring in two variables over a field $K$ and let $m_R = (X,Y)$ be the graded maximal ideal of $R$. Let $\mathcal{Q}=\{Q_n\}_{n\in\mathbb{N}}$ be any graded family of ideals in $R$ such that 
\begin{align*}
Q_0 &= R,\\
Q_1 &= \left(XY\right),\\
\left(XY^{n^2}, X^{n^2}Y\right) &\subset Q_n \subset \left(XY\right) \quad \forall n\geq 2.
\end{align*}
Then $$\limsup\limits_{n\to\infty} \dfrac{l_R\left(H^0_{m_R}\left(R/Q_n\right)\right)}{n^2\ln n} \leq 2.$$
\end{corollary}
\begin{proof}
Recall the family of monomial ideals $\mathcal{I} = \{I_n\}_{n\in\mathbb{N}}$ as defined in Example \ref{limit}. Using induction it can be shown that $I_n \subset J_n$ for all $n\geq 0$. Then $$0 \to H^0_{m_R}\left(\dfrac{J_n}{I_n}\right) \to H^0_{m_R}\left(\dfrac{R}{I_n}\right) \to H^0_{m_R}\left(\dfrac{R}{J_n}\right) \to H^1_{m_R}\left(\dfrac{J_n}{I_n}\right) \to \cdots$$ is a natural long exact sequence of local cohomology modules. Also $J_n/I_n$ is an Artinian $R$-module since it is annihilated by the $m_R$-primary ideal $\left(X^{n^2-1}, Y^{n^2-1}\right)$. Therefore $H^1_{m_R}\left(J_n/I_n\right) = 0$, which shows that $$0 \to H^0_{m_R}\left(\dfrac{J_n}{I_n}\right) \to H^0_{m_R}\left(\dfrac{R}{I_n}\right) \to H^0_{m_R}\left(\dfrac{R}{J_n}\right) \to 0$$ is an exact sequence of finite length $R$-modules. So $$l_R\left(H^0_{m_R}\left(\dfrac{R}{J_n}\right)\right) = l_R\left(H^0_{m_R}\left(\dfrac{R}{I_n}\right)\right) - l_R\left(H^0_{m_R}\left(\dfrac{J_n}{I_n}\right)\right) \leq l_R\left(H^0_{m_R}\left(\dfrac{R}{I_n}\right)\right).$$ It now follows from Example \ref{limit} that $$\limsup_{n\to\infty}\dfrac{l_R\left(H^0_{m_R}\left(R/Q_n\right)\right)}{n^2\ln{n}}\leq \lim_{n\to\infty}\dfrac{l_R\left(H^0_{m_R}\left(R/I_n\right)\right)}{n^2\ln{n}} = 2.$$
\end{proof}
\bibliographystyle{plain}
\bibliography{References}

\begin{thebibliography}{10}

\bibitem{JO}
JO~Amao.
\newblock On a certain hilbert polynomial.
\newblock {\em Journal of the London Mathematical Society}, 2(1):13--20, 1976.

\bibitem{BA}
Carles Bivi\`{a}-Ausina.
\newblock The analytic spread of monomial ideals.
\newblock {\em Communications in Algebra}, 31(7):3487--3496, 2003.

\bibitem{montano}
Yairon {Cid-Ruiz} and Jonathan {Monta{\~n}o}.
\newblock {Mixed multiplicities of graded families of ideals}.
\newblock {\em arXiv e-prints}, page arXiv:2010.11862, October 2020.

\bibitem{DC5}
Steven~Dale Cutkosky.
\newblock Asymptotic growth of saturated powers and epsilon multiplicity.
\newblock {\em Math. Res. Lett.}, 18:93--106, 2011.

\bibitem{DC6}
Steven~Dale Cutkosky.
\newblock Asymptotic multiplicities of graded families of ideals and linear
  series.
\newblock {\em Advances in Mathematics}, 264:55--113, 2014.

\bibitem{DC1}
Steven~Dale Cutkosky.
\newblock Asymptotic multiplicities.
\newblock {\em Journal of Algebra}, 442:260--298, 2015.

\bibitem{DC3}
Steven~Dale Cutkosky, Huy~T{\`a}i H{\`a}, Hema Srinivasan, and Emanoil
  Theodorescu.
\newblock Asymptotic behavior of the length of local cohomology.
\newblock {\em Canadian Journal of Mathematics}, 57(6):1178--1192, 2005.

\bibitem{DC4}
Steven~Dale Cutkosky, J{\"u}rgen Herzog, and Hema Srinivasan.
\newblock Asymptotic growth of algebras associated to powers of ideals.
\newblock {\em Mathematical Proceedings of the Cambridge Philosophical
  Society}, 148(1):55--72, 2010.

\bibitem{DM}
Hailong Dao and Jonathan Monta{\~n}o.
\newblock Length of local cohomology of powers of ideals.
\newblock {\em Transactions of the American Mathematical Society}, 371, 05
  2019.

\bibitem{das}
Suprajo {Das}.
\newblock {Epsilon multiplicity for graded algebras}.
\newblock {\em Journal of Pure and Applied Algebra}, 225(10), 2021.

\bibitem{HHRT}
Manfred Herrmann, Eero Hyry, J\"{u}rgen Ribbe, and Zhongming Tang.
\newblock Reduction numbers and multiplicities of multigraded structures.
\newblock {\em Journal of Algebra}, 197(2):311--341, 1997.

\bibitem{J}
J{\"u}rgen Herzog, Tony~J Puthenpurakal, and Jugal~K Verma.
\newblock Hilbert polynomials and powers of ideals.
\newblock {\em Mathematical Proceedings of the Cambridge Philosophical
  Society}, 145(3):623--642, 2008.

\bibitem{HS}
Craig Huneke and Irena Swanson.
\newblock {\em Integral closure of ideals, rings, and modules}, volume~13.
\newblock Cambridge University Press, 2006.

\bibitem{JM}
Jack Jeffries and Jonathan Monta\~{n}o.
\newblock The $j$-multiplicity of monomial ideals.
\newblock {\em Mathematical Research Letters}, 20:729--744, 12 2013.

\bibitem{Mat1}
Hideyuki Matsumura.
\newblock {\em Commutative algebra}, volume~56.
\newblock Addison Wesley Longman, 1970.

\bibitem{IS}
Irena Swanson.
\newblock Powers of ideals.
\newblock {\em Mathematische Annalen}, 307(2):299--313, 1997.

\bibitem{TAK}
Yukihide Takayama.
\newblock Combinatorial characterizations of generalized cohen-macaulay
  monomial ideals.
\newblock {\em Bull. Math. Soc. Sci. Math. Roumanie (N.S.)}, 48
  (96)(3):327--344, 2005.

\bibitem{BJ2}
Bernd Ulrich and Javid Validashti.
\newblock Numerical criteria for integral dependence.
\newblock {\em Mathematical Proceedings of the Cambridge Philosophical
  Society}, 151(1):95--102, 2011.

\bibitem{KW}
Kevin {Woods}.
\newblock {Presburger arithmetic, rational generating functions and
  quasi-polynomials}.
\newblock {\em Journal of Symbolic Logic}, 80(2):433--449, 2015.

\end{thebibliography}
\end{document}